\documentclass[10pt]{amsart}
\usepackage{amsmath,amssymb,setspace}
\usepackage{indentfirst}
\usepackage[small,nohug,heads=vee]{diagrams}
\diagramstyle[labelstyle=\scriptstyle]

\newtheorem{theorem}{Theorem}
\newtheorem{corollary}[theorem]{Corollary}

\newtheorem{proposition}[theorem]{Proposition}

\DeclareMathOperator{\res}{Res}
\DeclareMathOperator{\ind}{Ind}

\DeclareMathOperator{\ints}{\mathbb{Z}}

\DeclareMathOperator{\rep}{\bold{Rep-}}
\DeclareMathOperator{\gr}{\text{Groth}}

\begin{document}

\title{Induction/Restriction Bialgebras for Restricted Wreath Products}\author{Seth Shelley-Abrahamson}\date{August 2014}\maketitle

\begin{abstract} To a finite group $G$ one can associate a tower of wreath products $S_n \rtimes G^n$.  It is well known that the graded direct sum of the Grothendieck groups of the categories of finite dimensional complex representations of these groups can be given the structure of a graded Hopf algebra, and in fact a positive self-adjoint Hopf algebra in the sense of Zelevinsky [1], using the induction product and restriction coproduct.  This paper introduces and explores an analogously defined algebra/coalgebra structure associated to a more general class of towers of groups, obtained as a certain family of subgroups of wreath products in the case $G$ is abelian.  We call these groups restricted wreath products, and they include the infinite family of complex reflection groups $G(m, p, n)$.  It is known that in the case of full wreath products the associated Hopf algebra decomposes as a tensor power of the Hopf algebra of integral symmetric functions.  In the case of restricted wreath products, the associated algebra/coalgebra is no longer a Hopf algebra, but here we see that it contains an algebra containing every irreducible representation as a constituent and which is isomorphic to a tensor power of such an algebra/coalgebra associated to a smaller restricted wreath product, generalizing the tensor product decomposition for the full wreath products.  We closely follow the approach of [1].

\end{abstract}

\newpage

\tableofcontents

\section{Introduction} Let $G$ be a finite abelian group.  We may then construct the wreath product $S_n[G] := S_n \rtimes G^n$ as the group of monomial matrices with all nonzero entries in $G$.  As $G$ is abelian there is a surjection $S_n[G] \rightarrow G$ by taking the sum of the elements in $G$ appearing in a matrix.  For a subgroup $H \subset G$ let $G_n(G, H)$ denote the kernel of the composition $S_n[G] \rightarrow G \rightarrow G/H$, so that $G_n(G, G) = S_n[G]$ and $G_n(G, H)$ is the group of monomial matrices with entries in $G$ whose entries sum to an element of $H$.

Note that when $G$ is cyclic this construction yields the finite complex reflection groups in the family $G(m, p, n)$ where $p$ divides $m$.  Specifically, we have $G(m, p, n) = G_n(\ints/m, p\ints/m)$.

Let $R_0(G, H) = \ints$ and for $n > 0$ let $R_n(G, H) = K_0(\bold{Rep}-G_n(G, H))$ denote the Grothendieck group of the category of finite dimensional complex representations of $G_n(G, H)$.  We then construct the graded abelian group $$R(G, H) = \bigoplus_{n \geq 0} R_n(G, H)$$ which has a designated graded basis given by the isomorphism classes of irreducible representations (along with $1 \in \ints$ in degree 0) along with a graded nondegenerate symmetric bilinear form given by the usual pairing of representations.  This form will be denoted $\langle \cdot, \cdot \rangle$.  Note that the irreducible elements form a graded orthonormal basis for $R(G, H)$.

Using induction and restriction, one can place graded product and coproduct structures on $R(G, H)$.  In particular, we have an embedding of groups $G_k(G, H) \times G_l(G, H) \subset G_{k + l}(G, H)$ by the block-diagonal embedding of matrices, so we have an induction functor $$\ind : \rep(G_k(G, H) \times G_l(G, H)) \rightarrow \rep(G_{k + l}(G, H))$$ and a restriction functor $$\res: \rep(G_{k + l}(G, H)) \rightarrow \rep(G_k(G, H) \times G_l(G, H)).$$  These are exact functors, so we obtain maps at the level of Grothendieck groups: $$m_{k, l} : R_k(G, H) \otimes R_l(G, H) \rightarrow R_{k + l}(G, H)$$ $$m_{k, l}^*: R_{k + l}(G, H) \rightarrow R_k(G, H) \otimes R_l(G, H)$$ in view of the natural isomorphism $$R_k(G, H) \otimes R_l(G, H) \cong \gr(\rep(G_k(G, H) \times G_l(G, H))).$$  For $k = 0$ or $l = 0$ just set $m_{k, l}$ and $m^*_{k, l}$ to be the maps given by the natural isomorphism $R \otimes \ints \cong R$.  Set $$m = m_{G, H} =  \sum_{k, l \geq 0} m_{k, l} : R(G, H) \otimes R(G, H) \rightarrow R(G, H)$$ $$m^* = m^*_{G, H} = \sum_{k, l \geq 0} m_{k, l}^* : R(G, H) \rightarrow R(G, H) \otimes R(G, H).$$  It is immediate that $m_{G, H}$ gives $R(G, H)$ the structure of a graded commutative algebra with unit and that $m_{G, H}^*$ gives $R(G, H)$ the structure of a graded cocommutative coalgebra with counit.  Furthermore, by Frobenius reciprocity $m_{G, H}$ and $m^*_{G, H}$ are adjoint operators with respect to the inner product on $R(G, H)$ and the induced graded inner product on $R(G, H) \otimes R(G, H)$.  As they arise from functors, they send irreducible elements to nonzero sums of irreducible elements with nonnegative coefficients.

\section{A Tensor Product Subalgebra} In this section we will construct a natural positive injective map of algebras $$\Phi: \bigotimes_{l \in H^*} R(G/H, 1) \hookrightarrow R(G, H)$$ and we will see a weak form of surjectivity in the sense that every irreducible element in $R(G, H)$ occurs as a constituent of some element of the image of this map.  For $H = G$, the case of usual wreath products, $R(G/G, 1) = R(1, 1)$ is the Hopf algebra of integral symmetric functions and the injection above is the usual isomorphism of Hopf algebras known in that case.

Let $\phi \colon G_n(G, H) \rightarrow G_n(G/H, 1)$ be the map given by reducing the matrix entries mod $H$.  This gives rise to an exact sequence $$0 \rightarrow H^n \rightarrow G_n(G, H) \rightarrow G_n(G/H, 1) \rightarrow 0$$ where the first map is the diagonal embedding.  We obtain an additive functor $\phi^* : \rep(G_n(G/H, 1)) \rightarrow \rep(G_n(G, H))$ by pullback, which gives rise to a graded operator $\phi^* \colon R(G/H, 1) \rightarrow R(G, H)$.  As $\phi$ is surjective this map sends distinct irreducibles to distinct irreducibles, and in particular is an embedding of graded free abelian groups.

For $l \in H^*$, $l^{\otimes n}$ is a linear character of $H^n$ centralized by $G_n(G, H)$, so the action can be extended to all of $G_n(G, H)$ trivially with respect to the above exact sequence.  We then obtain an additive functor $\tau_l \colon \rep(G_n(G, H)) \rightarrow \rep(G_n(G, H))$ by tensoring with $l^{\otimes n}$, giving rise to a graded operator $\tau_l$ on $R(G, H)$.  These operators have several nice properties.  We see $\tau_l \circ \tau_{l'} = \tau_{ll'}$.  In view of the inner product on $R(G, H)$ in terms of characters, we see $\tau_l^* = \tau_{\bar{l}} = \tau_{l^{-1}} = \tau_l^{-1}$, so $\tau_l$ is an orthogonal operator.  It is clear that $\tau_l$ is a map of coalgebras, but then $\tau_l^* = \tau_{\bar{l}}$ is also a map of coalgebras, so since the form on $R(G, H)$ is nondegenerate we conclude $\tau_l$ is also a map of algebras.  In summary, the rule $l \mapsto \tau_l$ gives an action of $H^*$ on $R(G, H)$ by positive orthogonal bialgebra automorphisms.

For $l \in H^*$, set $\Phi_l = \tau_l \circ \phi^* \colon R(G/H, 1) \rightarrow R(G, H)$.  We then have

\begin{proposition} $\Phi_l$ is an injective map of bialgebras sending irreducibles to irreducibles. \end{proposition}

\begin{proof} For the first statement, in view of the preceding comments about $\tau_l$ and $\phi^*$ we need only check that $\phi^*$ is a map of bialgebras.  It is obvious that $\phi^*$ is a map of coalgebras, and to establish that it is a map of algebras we need to check that the diagram $$\begin{diagram} R(G/H, 1) \otimes R(G/H, 1) &\rTo^{\phi^* \otimes \phi^*}& R(G, H) \otimes R(G, H)\\\dTo^{m_{G/H, 1}}&&\dTo^{m_{G, H}}\\R(G/H, 1) &\rTo^{\phi^*}&R(G, H)\end{diagram}$$ commutes.  For this just note that $\phi$ induces a bijection on the coset spaces $G_{k + l}(G, H)/(G_k(G, H) \times G_l(G, H))$ and $G_{k + l}(G/H, 1)/(G_k(G/H, 1) \times G_l(G/H, 1))$, and the commutativity of the diagram then follows immediately from the definition of induced characters. \end{proof}

\begin{proposition}  The sub-bialgebra $\Phi_l(R(G/H, 1))$ of $R(G, H)$ has a graded basis whose degree $n$ part consists of the isomorphism classes of those irreducible representations $\pi$ of $G_n(G, H)$ whose restriction to $H^n$ contains the irreducible constituent $l^{\otimes n}$.\end{proposition}

\begin{proof} From the construction and the previous proposition that $\Phi_l(R(G/H, 1))$, we need only check that any such $[\pi]$ is in the image of $\Phi_l$.  Note that the $l^{\otimes n}$-isotypic piece of $\pi|_{H^n}$ is actually a submodule for $G_n(G, H)$, so $\tau_l^{-1}\pi$ is an irreducible representation of $G_n(G, H)$ with trivial $H^n$-action, so has the structure of an irreducible $G_n(G/H, 1) = G_n(G, H)/H^n$-representation $\pi'$.  But then $\tau_{l}^{-1}\pi = \phi^*\pi'$ so $\pi = \Phi_l(\pi)$, as needed.\end{proof}

\begin{proposition}  The sub-bialgebras $\Phi_l(R(G/H, 1))$ are pairwise orthogonal. \end{proposition}

\begin{proof} If $\pi$ and $\sigma$ are irreducible representations of $G_n(G, H)$ which are $l^{\otimes n}$-isotypic and $l'^{\otimes n}$-isotypic, respectively, for $l \neq l'$, then we have $$\langle \pi, \sigma \rangle_{G_n(G, H)} \leq \langle \pi, \sigma\rangle_{H^n} = \deg(\pi)\deg(\sigma)\langle l^{\otimes n}, l'^{\otimes n}\rangle_{H^n} = 0$$ so $\langle \pi, \sigma \rangle = 0$, and in view of the previous proposition our claim follows.\end{proof}

Now for $l \in H^*$ define the graded operator $\Psi_l \colon R(G, H) \rightarrow R(G/H, 1)$ on the degree $n$ part by operator associated to the additive functor $\Psi_l \colon \rep(G_n(G, H)) \rightarrow \rep(G_n(G/H, 1))$ defined by $\Psi_l(\pi) = \hom_{H^n}(l^{\otimes n}, \pi)$.  The $G_n(G/H, 1)$-action is given by $g.A \mapsto \tilde{g}A\tilde{g}^{-1}$ for $A \in \hom_{H^n}(l^{\otimes n}, \pi)$ and $\tilde{g} \in G_n(G, H)$ any lift of $g \in G_n(G/H, 1)$.  This map $g.A$ does not depend on the choice of lifting of $g$ because $A$ commutes with the action of $H^n$.  Clearly $g.A \in \hom_{H^n}(l^{\otimes n}, \pi)$ so $\Psi_l(\pi)$ is indeed a $G_n(G/H, 1)$-representation, and clearly $\Psi_l$ is an additive functor.

\begin{proposition} $\Psi_l$ is a left adjoint for $\Phi_l$ as functors, and in particular the operators $\Psi_l$ and $\Phi_l$ are adjoints on the level of the Grothendieck groups with respect to the bilinear form.  As this form is nondegenerate and $\Phi_l$ is a bialgebra homorphism, $\Psi_l$ too is a bialgebra homomorphism.\end{proposition}

\begin{proof} This follows from tensor-hom adjunction.\end{proof} 

It is clear that $\Psi_l \circ \Phi_l$ is naturally isomorphic to the identity functor and that $\Phi_l \circ \Psi_l$ is naturally isomorphic to the functor $I_l$ on $\rep(G_n(G, H))$ given by projection to the $l^{\otimes n}$-isotypic piece for the $H^n$-action (recall this is always a $G_n(G, H)$-subrepresentation).  At the level of Grothendieck groups, we obtain:

\begin{proposition} $\Psi_l \circ \Phi_l \colon R(G/H, 1) \rightarrow R(G/H, 1)$ is the identity, and $$\Phi_l \circ \Psi_l \colon R(G, H) \rightarrow R(G, H)$$ is orthogonal projection onto the sub-bialgebra $\Phi_l(R(G/H, 1)).$\end{proposition}

We now define the map mentioned at the start of this section $$\Phi \colon \bigotimes_{l \in H^*} R(G/H, 1) \rightarrow R(G, H)$$ as the product of the maps $\Phi_l$.  Given an $|H^*|$-tuple $\lambda = (\lambda_l)_{l \in H^*}$ of nonnegative integers, let $l(\lambda)$ denote the number of nonzero parts.  Given irreducible representations $\pi_l$ of $G_{\lambda_l}(G/H, 1)$, let $\pi_\lambda = \bigotimes_{l \in H^*} \pi_l \in \bigotimes_{l \in H^*} R(G/H, 1).$

\begin{theorem} $\Phi$ is a graded, positive, injective map of algebras.  Every irreducible element of $R(G, H)$ occurs as a constituent of some element of the image of $\Phi$.  In particular, we have $$\langle \Phi(\pi_\lambda), \Phi(\sigma_\mu)\rangle = \delta_{\pi_\lambda, \sigma_\mu} [G : H]^{l(\lambda) - 1}.$$\end{theorem}

\begin{proof} It follows from previous results that $\Phi$ is a positive graded map of algebras.  By positivity, showing the given inner product formula will imply injectivity, so we start there, which is just an application of Mackey's double coset formula.  In particular, we have $$\begin{array} {lcl}&&\langle \Phi(\pi_\lambda), \Phi(\sigma_\mu)\rangle\\&=&\langle \ind_{G_\lambda(G, H)}^{G_n(G, H)} \bigotimes_{l \in H^*} \Phi_l(\pi_l), \ind_{G_\mu(G, H)}^{G_n(G, H)}\bigotimes_{l \in H^*} \Phi_l(\sigma_l)\rangle_{G_n(G, H)}\\&=&\sum_{\gamma \in G_\lambda\backslash G_n/G_\mu} \langle \bigotimes_{l \in H^*} \pi_l, (\bigotimes_{l \in H^*} \sigma_l)^\gamma\rangle_{G_\lambda \cap\gamma G_\mu \gamma^{-1}} \end{array}$$ Noting $H^n \subset G_\lambda \cap \gamma G_\mu \gamma^{-1}$, we have the bound $$\left\langle \bigotimes_{l \in H^*} \pi_l, \left(\bigotimes_{l \in H^*} \sigma_l\right)^\gamma\right\rangle_{G_\lambda \cap\gamma G_\mu \gamma^{-1}}$$ $$\leq \prod_{l \in H^*} \deg(\pi_l)\deg(\sigma_l) \left\langle \bigotimes_{l \in H^*} l^{\otimes \lambda_l}, \left(\bigotimes_{l \in H^*} l^{\otimes \mu_i}\right)^\gamma\right\rangle_{H^n}$$  As $G^n$ centralizes $H^n$, twisting by $\gamma$ amounts to just twisting by some element of $S_n$, permuting the tensor factors, and hence the final inner product is $0$ unless both $\lambda = \mu$ and $\bar{\gamma} \in S_n$ stabilizes $\lambda$.  In this case we can clearly take $\gamma$ to be diagonal, so to compute the original inner product we need only sum over a collection of diagonal double coset representatives.  From the definition of $\Phi_l$, we see that any diagonal element of $G_{\lambda_i}(G, G)$ centralizes $\Phi_l(\sigma_l)$, and we conclude that each term of the above Mackey sum involving a diagonal representative yields a contribution of $1$ to the sum, as in that case $G_\lambda \cap \gamma G_\mu \gamma^{-1} = G_\lambda$ and it is just an inner product of an irreducible representation of $G_\lambda$ with itself.  The number of classes of the double coset space which have a diagonal representative is clearly $[G: H]^{l(\lambda) - 1}$ - indeed they are formed by a choice of element of $G/H$ for each nonzero $\lambda_i \times \lambda_i$ block such that the entire sum is $1 \in G/H$.

For our weak form of surjectivity, let $\pi \in \rep(G_n(G, H))$ be an irreducible representation.  The action of $S_n$ on $\pi$ induces an action of $S_n$ on the set of $H^n$-isotypic pieces by permuting the tensor factors, so we conclude some there is a nonzero $H^n$-isotypic piece of $\pi|_{H^n}$ of type $l_1^{\otimes \lambda_1} \otimes \cdots \otimes l_{|H|}^{\otimes \lambda_{|H|}}$. Then, $\pi|_{G_\lambda}$ contains some irreducible representation $\pi_1 \otimes \cdots \otimes \pi_{|H|}$ whose restriction $\pi_i|_{H^{\lambda_i}}$ contains $l_i^{\otimes \lambda_i}$.  By Proposition 2, we then have $\pi_i = \Phi_{l_i}(\pi_i')$ for some $\pi_i'$.  But then by Frobenius reciprocity $\pi$ is an irreducible constituent of $\Phi(\pi_\lambda)$, as needed.\end{proof}

Writing $$\Phi = m_{G, H}^{(|H^*|)} \circ \bigotimes_{l \in H^*} \Phi_l : \bigotimes_{l \in H^*} R(G/H, 1) \rightarrow R(G, H)$$ we also have the adjoint map $$\Psi = \bigotimes_{l \in H^*} \Psi_l \circ m_{G, H}^{*(|H^*|)} : R(G, H) \rightarrow \bigotimes_{l \in H^*} R(G/H, 1)$$ where $m_{G, H}^{(|H^*|)}$ and $m_{G, H}^{*(|H^*|)}$ denote iterated multiplication/comultiplication.  The first part of the previous theorem can then be recast as

\begin{corollary} $\Psi$ is a positive, graded, surjective map of coalgebras.  No positive element lies in its kernel. \end{corollary}

Note that in the case of usual wreath products, i.e. $G = H$, we have $[G : H] = 1$, so by the inner product formula in Theorem 6 we have that $\Phi$ sends irreducibles to irreducibles, and the weak surjectivity condition becomes usual surjectivity, so $\Phi$ is surjective and hence is a bijective isometry, so $\Phi^{-1} = \Phi^* = \Psi$.  But $\Psi$ is a map of coalgebras, so $\Phi = \Psi^{-1}$ is as well, and we obtain

\begin{corollary} For the case of usual wreath products, i.e. $G = H$, we have that $\Phi$ and $\Psi$ are mutually inverse and adjoint, positive (irreducible-to-irreducible), graded isomorphisms respecting both the algebra and coalgebra structures.  \end{corollary}

This case recovers the usual identification of $R(G, G)$ with the $|G|$-fold tensor power of the Hopf algebra of integral symmetric functions, as seen for example in [1].

\vfil\eject

\section {References}

1.  A. Zelevinsky.  \emph{Representations of Finite Classical Groups, A Hopf algebra approach,} volume 869 of \emph{Lecture Notes in Mathematics.} Springer-Verlag, Berlin, 1981.

\end{document}